\documentclass{amsart}
\usepackage{graphicx}
\usepackage{amscd}
\usepackage{amsmath}
\usepackage{amsfonts}
\usepackage{amssymb}
\theoremstyle{plain}
\newtheorem{theorem}{Theorem}[section]
\newtheorem{corollary}[theorem]{Corollary}
\newtheorem{lemma}[theorem]{Lemma}
\newtheorem{proposition}[theorem]{Proposition}
\theoremstyle{definition}

\newtheorem{definition}[theorem]{Definition}

\newtheorem{remark}[theorem]{Remark}
\theoremstyle{remark}

\begin{document}
\title{A Schreier domain type condition II}

\author{Zaheer Ahmad, Tiberiu Dumitrescu and Mihai Epure}
\address{Abdus Salam School of
Mathematical Sciences,
GC University Lahore, 68-B New
Muslim Town, Lahore 54600, Pakistan
}\email{zaheeir@yahoo.com}

\address{Facultatea de Matematica si Informatica,
University of Bucharest,
14 A\-ca\-de\-mi\-ei Str., Bucharest, RO 010014,
Romania}
\email{tiberiu@fmi.unibuc.ro, tiberiu\_dumitrescu2003@yahoo.com}

\address{Simion Stoilow Institute of Mathematics of the Romanian Academy
Research unit 5, P. O. Box 1-764, RO-014700 Bucharest, Romania}
\email{epuremihai@yahoo.com}

\thanks{2000 Mathematics Subject Classification: Primary 13A15, Secondary 13F05.}
\thanks{Key words and phrases: Schreier domain, Star operation,  Pr\"ufer domain.}

\begin{abstract}
\noindent
For an integral domain $D$ and a star operation $*$ on $D$, we study the  following condition: whenever $I \supseteq AB$ with $I$, $A$, $B$  nonzero ideals, there exist nonzero ideals $H$ and $J$ such that $I^{*}=(HJ)^{*}$, $H^{*}\supseteq A$ and $J^{*}\supseteq B$.
\end{abstract}

\maketitle
%
%\thanks{The authors gratefully acknowledge presentation of this paper.}
%

\section{Introduction.}

In \cite{C}, Cohn introduced  the  notion of Schreier domain.  A domain $D$ is said to be a {\em Schreier domain} if $(1)$ $D$ is integrally closed and $(2)$  whenever $I, J_1, J_2$ are principal ideals of $D$ and $I\supseteq J_1J_2$, then  $I = I_1I_2$ for some principal ideals $I_1,I_2$ of $D$ with $I_i \supseteq J_i$ for $i = 1,2$.  The study of Schreier domains was continued in \cite{MR} and \cite{Zps}. In  \cite{Zps}, a domain was called  a {\em pre-Schreier domain} if it satisfies  condition $(2)$ above.
Subsequently, extensions of the ``(pre)-Schreier domain'' concept were studied in \cite{DM}, \cite{ADZ}, \cite{DK}, \cite{AD} and \cite{ADE}.

In \cite{ADE}, we studied a class of domains that satisfies a Schreier-like condition for all ideals. More precisely, a domain $D$ was called  a {\em sharp domain } if whenever $I \supseteq AB$ with $I$, $A$, $B$  nonzero ideals of $D$, there exist  ideals $A'\supseteq A$ and $B'\supseteq B$ such that $I=A'B'$. We recall  several results from \cite{ADE}.
If the domain $D$ is  Noetherian or Krull, then $D$ is sharp if and only if $D$ is a  Dedekind domain \cite[Corollaries 2 and 12]{ADE}. A sharp domain is  pseudo-Dedekind; in particular, a sharp domain is a completely integrally closed GGCD domain \cite[Proposition 4]{ADE}.
Recall (cf. \cite{Z} and \cite{AK}) that a domain $D$ is called a {\em pseudo-Dedekind domain} (the name used in \cite{Z}  was {\em generalized Dedekind domain}) if the $v$-closure of each nonzero ideal of $D$ is invertible. Also, recall from \cite{AA} that a  domain $D$ is called a {\em generalized GCD domain (GGCD domain)} if  the $v$-closure of each nonzero finitely generated ideal of $D$ is invertible. The definition of the $v$-closure is recalled below.
A valuation domain is sharp if and only if the value group  of $D$ is  a complete subgroup of the reals \cite[Proposition 6]{ADE}.
The localizations of a  sharp domain at the maximal ideals are valuation domains with value group a complete subgroup of the reals; in par\-ti\-cu\-lar, a sharp domain is a Pr\"ufer domain of dimension at most one  \cite[Theorem 11]{ADE}.
The converse is true for the domains of finite character  \cite[Theorem 15]{ADE}, but not true in general \cite[Example 13]{ADE} (recall that a {\em domain of finite character} is a domain whose nonzero elements are contained in only finitely many maximal ideals).  A countable sharp domain is a Dedekind domain \cite[Corollary 17]{ADE}.

The purpose of this paper is to study the ``sharp domain''  concept in the  star operation setting. To facilitate the reading of the  paper, we first review some basic facts about $*$-operations. Let $D$ be a domain with quotient field $K$ and let $F(D)$ denote the set of nonzero fractional ideals of $D$.
A function       $A\mapsto A^*: F(D) \rightarrow F(D)$ is called a {\em star operation} on
$D$ if $*$ satisfies the following three conditions for all $0
\neq a \in K$ and all $I,J \in F(D)$:

$(1)$ $D^{*} = D$ and
$(aI)^{*}=aI^{*}$,

$(2)$ $I \subseteq I^{*}$ and if $ I
\subseteq J$, then $I^{*} \subseteq J^{*}$,

$(3)$
$(I^{*})^{*} = I^{*}.$

\noindent An ideal $I \in F(D)$ is called a $*$-ideal if $I = I^{*}.$
For all $I,J\in F(D)$, we have
$(IJ)^{*}=(I^{*}J)^{*}=(I^{*}J^{*})^{*}$.
These equations define the so-called {\em $*$-multiplication.}
If $\{I_\alpha\}$ is a subset of $F(D)$ such that $\cap I_\alpha\neq 0$, then
$\cap I_\alpha^*$ is a $*$-ideal.
Also, if $\{I_\alpha\}$ is a subset of $F(D)$ such that $\sum I_\alpha$ is a fractional ideal, then
$(\sum I_\alpha)^*=(\sum I_\alpha^*)^*$.
The star operation $*$ is said to be  {\em stable}  if  $(I \cap J)^{*} = I^{*}\cap J^{*}$ for all $I,J\in F(D)$.
If $*$ is a star operation, the function $*_f: F(D) \rightarrow F(D)$ given by $I^{*_f} = \cup_H H^{*}$, where $H$ ranges over all
nonzero finitely generated subideals of $I$, is also a star operation. The star operation $*$ is said to be of
{\em finite character} if $*=*_f$. Clearly $(*_f)_f=*_f$.
Denote by $Max_*(D)$  the set of maximal $*$-ideals,  that is, ideals maximal among proper integral
$*$-ideals of $D$. Every maximal $*$-ideal is a prime ideal.
The {\em $*$-dimension} of $D$ is
$sup \{ n\mid 0\subset P_1\subset \cdots \subset P_n,$  $P_i$  prime $*$-ideal of $D\}$.
Assume that $*$ is a star operation of finite character. Then every proper $*$-ideal is contained in some maximal $*$-ideal, and
the map $I\mapsto I^{\tilde{*}} = \cap _{P \in Max_*(D)}ID_P$ for all $I\in F(D)$ is a stable star operation of finite character, cf. \cite[Theorems 2.4, 2.5 and 2.9]{AC}.
Moreover, $*$ is stable if and only if $*=\tilde{*}$, cf. \cite[Corollary 4.2]{A}.
A $*$-ideal $I$ is of {\em finite type} if
$I=(a_1,...,a_n)^{*}$ for some $a_1,...,a_n\in I.$
A {\em Mori domain} is a domain whose $t$-ideals are of finite type (see \cite{B}). By \cite{HZ}, an integral domain is said to be a {\em TV domain} if every $t$-ideal is a $v$-ideal. A Mori domain is a TV domain.

A fractional ideal $I \in F(D)$ is said to be {\em $*$-invertible} if
$(II^{-1})^{*} = D$, where $I^{-1}=(D:I)=\{ x \in K \mid xI \subseteq D\}$.
If $*$ is of finite character, then $I$ is $*$-invertible if and only if
$II^{-1}$ is not contained in any maximal $*$-ideal of $D$; in this case $I^*=(a_1,...,a_n)^*$ for some $a_1,...,a_n\in I$.
%
% Also if $*$ is of finite character and $I$ is $*$-invertible,
%
Let $*_1,*_2$ be star operations on $D$. We write $*_1\leq *_2$, if $I^{*_1}\subseteq I^{*_2}$ for all
$I\in F(D)$. In this case we get $(I^{*_1})^{*_2}=I^{*_2}=(I^{*_2})^{*_1}$ and every
$*_1$-invertible ideal is $*_2$-invertible.
Some well-known star operations are: the {\em $d$-operation} (given by $I\mapsto I$),
the {\em $v$-operation} (given  by $I\mapsto I_v = (I^{-1})^{-1}$) and   the {\em $t$-operation} (defined by $t=v_f$).
The {\em $w$-operation} is the star operation given by $I \mapsto I_w= \{x\in K \mid xH \subseteq
I$ for some finitely
generated ideal $H$ of $D$ with $H^{-1} =D\}$. The $w$-operation is a stable star operation of finite character.
For   an integrally closed domain $D$, the {\em $b$-operation} on $D$ is the star operation  defined by $I\mapsto I_b=\cap _{V} IV$  where $V$ runs in the set of all valuation overrings of $D$ (see \cite[Page 398]{G}). For every $I\in F(D)$, we have
$I\subseteq I_w \subseteq I_t \subseteq I_v$. It is known that $Max_w(D)=Max_t(D)$, cf.
\cite[Corollaries 2.13 and 2.17]{AC} and $I_w = \cap _{M \in Max_t(D)}ID_M$, cf. \cite[Corollary 2.10]{AC}.
Consequently, a nonzero fractional ideal is $w$-invertible if and only if it is $t$-invertible. Recall \cite{EFP} that an integral domain $D$ is said to be {\em $*$-Dedekind} if every nonzero fractional ideal of $D$ is $*$-invertible. A domain $D$ is called a Prufer {\em $*$-multiplication domain (P$*$MD)} if every nonzero finitely generated ideal of $D$ is $*_f$-invertible (see \cite{FJS}).
For the general theory of star operations we refer the reader to \cite[Sections 32 and 34]{G}.\\

We introduce the key concept of this paper.

\begin{definition}\label{1}
Let $*$ be a star operation on $D$. We say that a domain $D$ is a {\em  $*$-sharp domain} if whenever $I$, $A$, $B$  are nonzero ideals of $D$ with $I \supseteq AB$, there exist nonzero ideals $H$ and $J$ such that $I^{*}=(HJ)^{*}$, $H^{*}\supseteq A$ and $J^{*}\supseteq B$.
\end{definition}

The $d$-sharp domains are just the sharp domains studied in \cite{ADE}. If $*_1 \leq *_2$ are star operations and $D$ is $*_1$-sharp, then $D$ is $*_2$-sharp (Proposition \ref{81}). In particular, if $*$ is a star operation, then  every sharp domain is $*$-sharp and every $*$-sharp domain is $v$-sharp. A $t$-sharp domain is $v$-sharp but the converse is not true in general (Remark \ref{121}).

In Section 2, we study the $*$-sharp domains in general.
In this new context, we generalize most of the results obtained in \cite{ADE}.
For  $*\in \{d,b,w,t\}$, every fraction ring of a $*$-sharp domain is $*$-sharp (Proposition \ref{81}).
%If D is $v$-sharp then $D_s$ is not $v$-sharp in general (see paragraph after Corollary \ref{100} ). I
Every $*$-Dedekind domain is $*$-sharp. In particular, every Krull domain is  $t$-sharp (Proposition \ref{3}).
Let $D$ be a domain and  $*$  a finite character  stable star operation such that $D$ is $*$-sharp. Then $D$ is a P$*$MD of $*$-dimension $\leq 1$; moreover   $D_M$ is a  valuation domain with value group a complete subgroup of the reals, for each  $M\in Max_*(D)$ (Proposition \ref{77}).
The converse is true for domains  whose  nonzero elements are contained in only finitely many $*$-maximal ideals (Proposition \ref{200}).
If  $*$ is a star operation on $D$ such that $D$ is a $*$-sharp domain, then $I_v$ is $*$-invertible for each nonzero ideal $I$
(Proposition \ref{3a}).
If $*$ is a finite character stable star operation  on $D$ such that $D$ is a $*$-sharp TV domain, then $D$ is $*$-Dedekind (Corollary \ref{111}).
A domain $D$ is $v$-sharp if and only if $D$ is completely integrally closed (Corollary \ref{100}). In particular, every $*$-sharp domain is completely integrally closed.
If  $*$ is a stable star operation on $D$ such that $D$ is a $*$-sharp domain, then every finitely generated nonzero ideal of $D$ is $*$-invertible (Proposition \ref{102}).
If $D$ is a countable domain and $*$ a finite character  stable star operation on $D$  such that $D$ is $*$-sharp, then $D$ is a $*$-Dedekind domain (Corollary \ref{845}).

In Section 3, we study the $t$-sharp domains. We obtain the following results.
Every $t$-sharp domain $D$ is a PVMD with $t$-dimension $\leq 1$ and $D_M$ is a  valuation domain with value group a complete subgroup of the reals, for each maximal $t$-ideal $M$ of $D$ (Proposition  \ref{11}). A domain is $t$-sharp if and only if it is  $w$-sharp  (Proposition \ref{103}).
A domain  $D$ is a Krull domain if and only if $D$  is a $t$-sharp TV domain (Corollary \ref{400}).
If $D$ is a countable $t$-sharp domain, then $D$ is a Krull domain
(Corollary \ref{300}).
A domain $D$ is $t$-sharp if and only if $D[X]$ is $t$-sharp
(Proposition \ref{331}) if and only if $D[X]_{N_v}$ is sharp (Proposition \ref{133}). Here  $N_v$ denotes the multiplicative subset of $D[X]$ consisting of all  $f\in D[X]-\{0\}$ with $c(f)_v=D$, where $c(f)$ is the ideal generated by the coefficients of $f$.
Let $D$ be a $t$-sharp domain. Then $N'_v=\{ f\in D[[X]]-\{0\}\mid c(f)_v=D\}$ is  a multiplicative set,
$D[[X]]_{N'_v}$ is a sharp domain and every ideal of $D[[X]]_{N'_v}$ is extended from $D$ (Proposition \ref{1024}). Moreover, $D[[X]]_{N'_v}$ is a faithfully flat $D[X]_{N_v}$-module
and there is a one-to-one correspondence between the ideals of $D[X]_{N_v}$ and the  ideals of $D[[X]]_{N'_v}$
(Corollary \ref{307}).

Throughout this paper all rings are (commutative unitary) integral domains. Any unexplained material is standard, as in \cite{G}, \cite{H}.

\section{$*$-sharp domains.}

In this section we study the $*$-sharp domains  for an arbitrary star operation $*$ (see Definition \ref{1}).  We obtain $*$-operation analogues for most of the results  in \cite{ADE}.

\begin{proposition} \label{4}
Let $D$ be a domain, $S\subseteq D$ a multiplicative set and $*$ (resp. $\sharp$)  star operations  on $D$ (resp. $D_S$) such that
$I^*\subseteq (ID_S)^\sharp$  for each nonzero ideal $I$ of $D$.
If $D$ is  $*$-sharp, then the fraction ring $D_S$  is  $\sharp$-sharp.
\end{proposition}
\begin{proof}
Note that the condition $I^*\subseteq (ID_S)^\sharp$ in the hypothesis is equivalent to $(I^*D_S)^\sharp=(ID_S)^\sharp$.
Let $I,A,B$ be nonzero ideals  of $D$ such that $ID_S\supseteq ABD_S$. Then $C=ID_S \cap D \supseteq  AB$. As $D$ is $*$-sharp, we have $C^*=(HJ)^*$ with $H,J$ ideals of $D$ such that  $H^*\supseteq A$ and $J^*\supseteq B$.  Since $(WD_S)^\sharp=(W^*D_S)^\sharp$ for every nonzero ideal $W$,
we get $(ID_S)^\sharp=(C^*D_S)^\sharp=((HJ)^*D_S)^\sharp=(HJD_S)^\sharp$,  $(HD_S)^\sharp=(H^*D_S)^\sharp\supseteq AD_S$ and  $(JD_S)^\sharp\supseteq BD_S$.
\end{proof}

%Let $D$ be an integrally closed domain. Recall that the {\em $b$-operation} on $D$ is  the star operation given by $I\mapsto I_b=\cap_V IV$ where $V$ runs in the set of all valuation overrings  of $D$ (see \cite[page 398]{G}).

\begin{proposition} \label{81}
Let $D$ be a domain, $*_1 \leq  *_2$   star operations and $D$ and $S\subseteq D$ a multiplicative set.

$(a)$  If $D$ is $*_1$-sharp, then D is $*_2$-sharp.

$(b)$ If $*\in \{d,t,w,b\}$ and $D$ is  $*$-sharp (with $D$ integrally closed if $*=b$), then  $D_S$  is  $*$-sharp.
\end{proposition}
\begin{proof}
$(a)$. Apply Proposition \ref{4} for $S=\{1\}$, $*=*_1$ and $\sharp=*_2$.
$(b)$. By Proposition \ref{4}, it suffices to show that $I^*\subseteq (ID_S)^*$  for each nonzero ideal $I$ of $D$. This is clear for $*=d$ and true for $*=t$, cf. \cite[Lemma 3.4]{Kg}. Assume that $x\in I_w$. Then $xH\subseteq I$ for some finitely generated nonzero ideal $H$ of $D$ such that $H_v=D$. Hence $(HD_S)_v=D_S$ (cf. \cite[Lemma 3.4]{Kg}) and $xHD_S\subseteq ID_S$, thus $x\in (ID_S)_w$. Assume that $D$ integrally closed. If $V$ is a valuation overring of $D_S$, then  $V$ is an overring of $D$, so $I_b\subseteq IV$. Thus $I_b\subseteq (ID_S)_b$.
\end{proof}

In \cite[Theorem 11]{ADE}, it was shown that a sharp domain is a Prufer domain of dimension at most $1$. We extend this result.

\begin{proposition} \label{77}
Let $D$ be a domain and $*$ a finite character  stable star operation on $D$  such that $D$ is $*$-sharp.  Then    $D_M$ is a  valuation domain with value group a complete subgroup of the reals, for each  $M\in Max_*(D)$. In particular, $D$ is a P$*$MD of $*$-dimension $\leq 1$.
\end{proposition}
\begin{proof}
Let $M$ be a maximal $*$-ideal. If  $I$ is a nonzero ideal  of $D$, then $I^*D_M=ID_M$, cf. \cite[Corollary 4.2]{A}. By   Proposition \ref{4}, applied for $S=D-M$, $*=*$ and $\sharp=d$, we get that $D_M$ is a sharp domain. Apply \cite[Theorem 11]{ADE}. The ``in particular'' assertion is clear.
\end{proof}

\begin{proposition}\label{3}
Let $D$ be a domain and $*$ a star operation on $D$. If $D$ is $*$-Dedekind, then $D$ is $*$-sharp. In particular, every Krull domain is   $t$-sharp.
\end{proposition}
\begin{proof}
Let $I,A,B$ be nonzero ideals of $D$ such that $I \supseteq AB$. Set $H=I+A$ and $J=IH^{-1}$. Note that $J\subseteq D$ and $A\subseteq H$. Since $(HH^{-1})^*=D$, we get $I^*=(HJ)^*$. From $BH=B(A+I) \subseteq I$, we get  $B \subseteq (BHH^{-1})^* \subseteq (IH^{-1})^*=J^*$. For the ``in particular statement'', recall that  the $t$-Dedekind domains are the Krull domains, cf. \cite[Theorem 3.6]{Kg1}.
\end{proof}

\begin{proposition}\label{3a}
Let $D$ be a domain and $*$ a star operation on $D$ such that $D$ is $*$-sharp. Then $I_v$ is $*$-invertible for each nonzero ideal $I$.
\end{proposition}
\begin{proof}
Let $I$ be a nonzero ideal of $D$ and $x  \in I-\{0\}$. Then $I(xI^{-1}) \subseteq xD$. Since $D$ is $*$-sharp, there exist $H,J$ ideals of $D$ such that  $H^* \supseteq I$,  $J^* \supseteq xI^{-1}$ and  $xD=(HJ)^*$. Hence $H$ is $*$-invertible and we get
$H^{-1}=(x^{-1}J)^*\supseteq (xx^{-1}I^{-1})^*=I^{-1}$, so $H_v\subseteq I_v$. The opposite inclusion follows from  $H^* \supseteq I$. Thus $I_v=H_v$ is $*$-invertible, because $H^*=H_v$ since $H$ is $*$-invertible.
\end{proof}

Next, we extend  \cite[Corollary 12]{ADE} to the star operation setting.

\begin{corollary}\label{111}
Let $D$ be a domain and $*$ a finite character  stable star operation on $D$  such that $D$ is $*$-sharp. If $D$ is a TV domain (e.g. a Mori domain), then $D$ is $*$-Dedekind.
\end{corollary}
\begin{proof}
By Proposition \ref{77}, $D$ is a P$*$MD , so $*=t$, cf. \cite[Proposition 3.15]{FJS}. As $D$ is a TV domain, we get $*=t=v$. By Proposition \ref{3a}, $D$ is $*$-Dedekind.
\end{proof}

\begin{corollary}\label{100}
A domain $D$ is $v$-sharp if and only if $D$ is completely integrally closed. In particular, any  $*$-sharp domain  is  completely integrally closed.
\end{corollary}
\begin{proof}
By Propositions \ref{3} and \ref{3a} (for $*=v$),  $D$ is $v$-sharp if and only if $D$ is $v$-Dedekind. By \cite[Theorem 34.3]{G} or \cite[Proposition 3.4]{F}, a domain is $v$-Dedekind if and only if it is  completely integrally closed. For the ``in particular'' assertion, apply  Proposition \ref{81} taking into account that $*\leq v$ for each star operation $*$.
\end{proof}

% CONSERVARE            \begin{proposition} \label{4}
%If $D$ is a $t$-sharp domain, then every fraction ring $D_S$ of $D$ is also $t$-sharp.\end{proposition}\begin{proof}Let $I,A,B$ be nonzero ideals  of $D$ such that $ID_S\supseteq ABD_S$. Then $C=ID_S \cap D \supseteq  AB$. As $D$ is $t$-sharp, we have $C_t=(HJ)_t$ with $H,J$ ideals of $D$ such that  $H_t\supseteq A$ and $J_t\supseteq B$.  Since $(WD_S)_t=(W_tD_S)_t$ for every nonzero ideal $W$ (see ???),we get $(ID_S)_t=(C_tD_S)_t=((HJ)_tD_S)_t=(HJD_S)_t$,  $(HD_S)_t=(H_tD_S)_t\supseteq AD_S$ and  $(JD_S)_t\supseteq BD_S$.\end{proof}

\begin{remark}\label{121}
$(a)$ There exist  a completely integrally closed domain $A$  having some fraction ring which is not completely integrally closed (for instance the ring of entire functions, cf. \cite[Exercises 16 and 21, page 147]{G}). Thus the $v$-sharp property does not localize, cf. Corollary \ref{100}. Note that $A$ cannot be $t$-sharp because the $t$-sharp property localizes, cf.
Proposition \ref{81}.
$(b)$ Let $D$ be a completely integrally closed domain which is not a PVMD (such a domain is constructed in  \cite{D}). By  Corollary \ref{100} and Proposition  \ref{11}, such a domain is $v$-sharp but not $t$-sharp.
$(c)$ Let $D$ be a Krul domain of dimension $\geq 2$ (e.g. $\mathbb{Z}[X]$).  By  Proposition \ref{3} and \cite[Theorem 11]{ADE}, $D$ is  $t$-sharp but not sharp.

\end{remark}

In the next lemma we recall two well-known facts.

\begin{lemma} \label{71}
Let $D$ be a domain, $*$ a  star operation on $D$ and $I,J,H\in F(D)$.

$(a)$ If $(I+J)^*=D$, then $(I\cap J)^*=(IJ)^*$.

$(b)$ If $I$ is $*$-invertible, then $(I(J\cap H))^*=(IJ\cap IH)^*$.
\end{lemma}
\begin{proof}
$(a)$ Clearly,  $(IJ)^*\subseteq (I\cap J)^*$. Conversely, since  $(I+J)^*=D$, we have $(I\cap J)^*=((I\cap J)(I+J))^*\subseteq (IJ)^*$, thus $(I\cap J)^*=(IJ)^*$.
$(b)$ Clearly, $(I(J\cap H))^*\subseteq (IJ\cap IH)^*$. Conversely, because $I$ is $*$-invertible, we have $(IJ\cap IH)^*=(II^{-1}(IJ\cap IH))^*\subseteq (I(I^{-1}IJ\cap I^{-1}IH))^*\subseteq (I(J\cap H))^*$.
\end{proof}

The next result generalizes \cite[Proposition 10]{ADE}.

\begin{proposition} \label{211p}
Let $D$ be a domain and $*$ a stable star operation on $D$  such that $D$ is $*$-sharp.  If $I,J$ are nonzero ideals of $D$ such that $(I+J)_v=D$, then $(I_v+J_v)^*=D$.
\end{proposition}
\begin{proof}
Let $K$ be the quotient field of $D$. Changing $I$ by $I_v$ and $J$ by $J_v$, we may assume that $I,J$ are $*$-invertible $v$-ideals, cf. Proposition \ref{3a}.
Since $(I+J)^2 \subseteq I^2+J$ and $D$ is $*$-sharp, there exist two nonzero ideals $A$, $B$ such that $(I^2+J)^*=(AB)^*$ and $I+J\subseteq A^* \cap B^*$. We {\em claim} that $(I^2+J)^*:I=(I+J)^*$. To prove the claim, we perform the following step-by-step computation. First,
$(I^2+J)^*:I=((I^2+J)^*:_KI)\cap D=((I^2+J)I^{-1})^*\cap D=(I+JI^{-1})^*\cap D$ because $I$ is $*$-invertible. As $*$ is stable, we get $(I+JI^{-1})^*\cap D=((I+JI^{-1})\cap D)^*=(I+(JI^{-1}\cap D))^*$ by modular distributivity.
Since $I$ is $*$-invertible, we get
$(I+(JI^{-1}\cap D))^*=(I+I^{-1}(J\cap I))^*$, cf. Lemma \ref{71}.
Using the fact that $I$ is $*$-invertible (hence $v$-invertible) and Lemma \ref{71}, we derive that $(I+I^{-1}(J\cap I))^* \subseteq (I+I^{-1}(IJ)_v)^*\subseteq  (I+(II^{-1}J)_v)^*=(I+J_v)^*=(I+J)^*$. Putting all these facts together, we get $(I^2+J)^*:I\subseteq (I+J)^*$ and the other inclusion is clear. So the claim is proved.
From $(I^2+J)^*=(AB)^*$, we get $A^*\subseteq (I^2+J)^*:B^*\subseteq (I^2+J)^*:I=(I+J)^*$, so $A^*=(I+J)^*$. Similarly, we get $B^*=(I+J)^*$, hence $(I^2+J)^*=((I+J)^2)^*$.
It follows that $J^*\subseteq (I^2+J)^*=((I+J)^2)^*\subseteq  (J^2+I)^*$. So $J^*=J^*\cap (J^2+I)^*=(J\cap (J^2+I))^*=(J^2+(J\cap I))^*$ where we have used the fact that $*$ is stable and the modular distributivity. By Lemma \ref{71}, we have $I\cap J\subseteq (IJ)_v$, so we get $J^*=(J^2+(J\cap I))^*\subseteq (J^2+(IJ)_v)^*$. Since $J$ is $*$-invertible, we have $D=(JJ^{-1})^*\subseteq ((J^2+(IJ)_v)J^{-1})^* \subseteq (J+I_v)^*=(J+I)^*$.
Thus $(I+J)^*=D$.
\end{proof}

Note that from Proposition \ref{211p} we can recover easily  \cite[Proposition 10]{ADE}. Next, we give another extension of \cite[Theorem 11]{ADE} (besides Proposition \ref{77}).

\begin{proposition} \label{102}
Let $D$ be a domain and $*$ a stable star operation on $D$  such that $D$ is $*$-sharp. Then every finitely generated nonzero ideal of $D$ is $*$-invertible.
\end{proposition}
\begin{proof}
Let $x,y\in D-\{0\}$. By Proposition \ref{3a}, the ideal $I=(xD+yD)_v$ is $*$-invertible (hence $v$-invertible), so $(xI^{-1}+yI^{-1})_v=D$. By Proposition \ref{211p} we get $(xI^{-1}+yI^{-1})^*=D$, hence $I=((xI^{-1}+yI^{-1})I)^*=(xD+yD)^*$ because $I$ is $*$-invertible. Thus every two-generated nonzero ideal of $D$ is $*$-invertible. Now  the proof of \cite[Proposition 22.2]{G} can be easily adapted to show that every finitely generated nonzero ideal of $D$ is $*$-invertible.
\end{proof}

\begin{remark}
Under the assumptions  of Proposition \ref{102}, it does not follow that $D$ is a P$*$MD. Indeed, let $D$ be a completely integrally closed domain which is not a PVMD (such a domain is constructed in  \cite{D}).  The $v$-operation on $D$ is stable (cf. \cite[Theorem 2.8]{ACl}) and   $D$ is $v$-sharp (cf. Corollary \ref{100}).
\end{remark}

Let $D$ be a domain with quotient field $K$.
According to \cite{AZ}, a family   $\mathcal{F}$ of nonzero prime ideals of $D$ is called    {\em independent of finite character family (IFC family)},
if $(1)$ $D=\cap_{P\in \mathcal{F}}D_P$, $(2)$ every nonzero $x\in D$ belongs to only finitely many members of
$\mathcal{F}$ and $(3)$ every nonzero prime ideal of $D$ is contained in at most one member of $\mathcal{F}$. The follwing result extends \cite[Theorem 15]{ADE}. %We denote by $Max_*(D)$ the set of all maximal $*$-ideals of $D$.

\begin{proposition}\label{200}
Let $D$ be a domain and $*$ a finite character star operation on $D$. Assume that

$(a)$ every $x\in D-\{0\}$ is contained in only finitely many maximal $*$-ideals, and

$(b)$ for every $M\in Max_*(D)$, $D_M$ is a  valuation domain with value group a complete subgroup of the reals.
\\ Then $D$ is a $\tilde{*}$-sharp domain and hence $*$-sharp.
\end{proposition}
\begin{proof}
By \cite{Gr}, $D=\cap_M D_M$ where $M$ runs in the set of maximal $*$-ideals. Since each $D_M$ is a  valuation domain with value group a complete subgroup of the reals, every $M$ has height one.
It follows that $Max_*(D)$  is an IFC family. Consider the $\tilde{*}$ operation, i.e. $I\mapsto I^{\tilde{*}}=\cap_M ID_M$. We show that $D$ is $\tilde{*}$-sharp.
Let $I,A,B$ be nonzero ideals  of $D$ such that $I\supseteq AB$.
Let $P_1$,...,$P_n$ the maximal $*$-ideals of $D$ containing $AB$.
Since $D_{P_i}$ is sharp, there exist $H_i$, $J_i$ ideals of $D_{P_i}$ such that $ID_{P_i}=H_iJ_i$, $H_i\supseteq AD_{P_i}$ and
$J_i\supseteq BD_{P_i}$ for all $i$ between $1$ and $n$.
Set $H'_i=H_i\cap D$, $J'_i=J_i\cap D$, $i=1,...,n$,
$H=H'_1\cdots H'_n$ and $J=J'_1\cdots J'_n$.
By \cite[Lemma 2.3]{AZ}, $P_i$ is the only element of $Max_*(D)$ containing $H'_i$ (resp. $J'_i$), thus it can be checked that $ID_P=(HJ)D_P$, $HD_P\supseteq AD_P$ and $JD_P\supseteq BD_P$ for each $P\in Max_*(D)$. So, we have $I^{\tilde{*}}=(HJ)^{\tilde{*}}$, $H^{\tilde{*}}\supseteq A$ and $J^{\tilde{*}}\supseteq B$. Consequently, $D$ is $\tilde{*}$-sharp. By Proposition \ref{81}, $D$ is  $*$-sharp because  $\tilde{*}\leq *$, cf. \cite[Theorem 2.4]{AC}.
\end{proof}

\begin{proposition}\label{822}
Let $D$ be a contable domain and $*$ a finite character star operation on $D$ such that $D$ is a P$*$MD and  $I_v$ is $*$-invertible for each nonzero ideal $I$ of $D$. Then  every nonzero element of $D$ is contained in only finitely many maximal $*$-ideals.
\end{proposition}
\begin{proof}
Deny. By \cite[Corollary 5]{DZ},  there exists a nonzero element $z$ and an infinite family $(I_n)_{n\geq 1}$ of $*$-invertible proper ideals containing $z$ which are mutually $*$-comaximal (that is, $(I_m+I_n)^*=D$ for every $m\neq n$). For each nonempty set  $\Lambda$ of natural numbers, consider the $v$-ideal $I_\Lambda=\cap_{n\in \Lambda} I_n$ (note that $z\in I_\Lambda$). By hypothesis, $I_\Lambda$ is $*$-invertible. We claim that $I_\Lambda\neq I_{\Lambda'}$ whenever $\Lambda$, $\Lambda'$ are distinct nonempty sets of natural numbers.
Deny. Then there exists a nonempty set  of natural numbers $\Gamma$ and some $k\notin \Gamma$ such that $I_k\supseteq I_\Gamma$. Consider the ideal  $H=(I_k^{-1} I_\Gamma)^*\supseteq I_\Gamma$. If $n\in \Gamma$, then $I_n \supseteq I_\Gamma =(I_kH)^*$, so $I_n \supseteq H$, because $(I_n+I_k)^*=D$. It follows that $I_\Gamma\supseteq H$, so $I_\Gamma = H=(I_k^{-1} I_\Gamma)^*$. Since $I_\Gamma$ is $*$-invertible, we get $I_k=D$, a contradiction. Thus the claim is proved. But then it follows that $\{I_\Lambda\mid \emptyset\neq \Lambda\subseteq \mathbb{N}\}$ is an uncountable set of $*$-invertible ideals. This leads to a contradiction, because $D$ being countable, it has countably many $*$-ideals of finite type.
\end{proof}

\begin{corollary}\label{845}
Let $D$ be a countable domain and $*$ a finite character  stable star operation on $D$  such that $D$ is $*$-sharp.
Then $D$ is a $*$-Dedekind domain.
\end{corollary}
\begin{proof}
We may assume that $D$ is not a field.
By Proposition \ref{77}, $D$ is a P$*$MD. Now Propositions \ref{3a} and \ref{822} show that every nonzero element of $D$ is contained in only finitely many maximal $*$-ideals. Let $M$ be a maximal $*$-ideal of $D$. By Proposition \ref{77}, $D_M$ is a countable valuation domain with value group $\mathbb{Z}$ or $\mathbb{R}$, so $D_M$ is a DVR. Thus $D$ is a $*$-Dedekind domain, cf. \cite[Theorem 4.11]{EFP}.
\end{proof}

\section{$t$-sharp domains.}

The $t$-operation is a very useful  tool in multiplicative ideal theory. In this section we give some results  which are specific for  the $t$-sharp domains.

\begin{proposition}\label{11}
Let $D$ be a $t$-sharp domain. Then $D$ is a PVMD of $t$-dimension $\leq 1$ and $D_M$ is a  valuation domain with value group a complete subgroup of the reals for each maximal $t$-ideal $M$ of $D$.
\end{proposition}
\begin{proof}
Let $I$ be finitely generated nonzero ideal of $D$. Then $I_v=I_t$, so Proposition \ref{3a} shows that $I_t$ is $t$-invertible. Thus $D$ is a PVMD.
Let $M$ be a maximal $t$-ideal  of $D$.
By  part $(b)$ of Proposition \ref{81},  $D_M$ is a  $t$-sharp valuation domain, so $D_M$ is sharp since for valuation domains $t=d$. Now apply \cite[Proposition 6]{ADE}.
\end{proof}

\begin{proposition}\label{103}
Let $D$ be a domain. Then $D$ is $t$-sharp if and only if $D$ is $w$-sharp.
\end{proposition}
\begin{proof}
If $D$ is $w$-sharp, then $D$ is $t$-sharp (cf. Proposition \ref{81}) because $w\leq t$. Conversely, assume that $D$ is $t$-sharp. By Proposition \ref{11}, $D$ is a PVMD. But in a PVMD the $w$-operation coincides with the $t$-operation (cf. \cite[Theorem 3.5]{Kg}),  so $D$ is also $w$-sharp.
\end{proof}

Combining  Corollary \ref{111} and Proposition \ref{103}, we get

\begin{corollary}\label{400}
A domain  $D$ is a Krull domain if and only if $D$  is a $t$-sharp TV domain.
\end{corollary}

\begin{corollary}\label{300}
If $D$ is a countable $t$-sharp domain, then $D$ is a Krull domain.
\end{corollary}
\begin{proof}
Assume that $D$ is a countable $t$-sharp domain.
By Proposition \ref{103},  $D$ is $w$-sharp. Moreover the $w$-operation is stable and of finite character, cf. \cite[Corollary 2.11]{AC}. By Corollary \ref{845}, $D$ is a $t$-Dedekind domain, that is a Krull domain.
\end{proof}

%We show that for a $(t,v)$-Dedekind domain $D$ it suffices to test the condition in Definition only for ideals $I$ with $I_v=D$.

By \cite{Zac}, a domain $D$ is called  a {\em pre-Krull domain} if $I_v$ is $t$-invertible for each nonzero ideal $I$ of $D$ (see also \cite{L} where a pre-Krull domain is called a $(t,v)$-Dedekind domain).

\begin{proposition} \label{3x}
A domain $D$ is  $t$-sharp   if and only if

$(a)$ $D$ is pre-Krull, and

$(b)$ for all nonzero ideals  $I$,$A$,$B$   of $D$ such that $I_v=D$ and $I\supseteq AB$, there exist nonzero ideals $H$ and $J$ such that $I_t=(HJ)_t$, $H_t\supseteq A$ and $J_t\supseteq B$.
\end{proposition}
\begin{proof}
The implication $(\Rightarrow)$ follows from Proposition \ref{3a}. Conversely, assume that $(a)$ and $b$ hold.
Let $I$, $A$ and $B$ be nonzero ideals of $D$ such that $I\supseteq AB$.
Then $I_v \supseteq A_vB_v$ and $I_v$, $A_v$, $B_v$ are $t$-invertible ideals, cf. $(a)$.
Since $D$ is a pre-Krull domain,  $D$ is a  PVMD and hence a $t$-Schreier domain cf. \cite[Corollary 6]{DZ1}.
So there exist  $t$-invertible ideals $H$ and $J$ such that $I_v=(HJ)_t$, $H_t\supseteq  A_v$ and $J_t\supseteq  B_v$.
We have $(II^{-1})_t = (IH^{-1}J^{-1})_t  \supseteq (AH^{-1})(BJ^{-1})$. Set $M=(II^{-1})_t$ and note that  $AH^{-1}$ and $BJ^{-1}$ are integral ideals.
Since $I_v$ is $t$-invertible, $M_v=(I_vI^{-1})_v=D$.
By $(b)$, there exist  nonzero ideals $N$ and $P$ such that $M_t=(NP)_t$, $N_t\supseteq AH^{-1}$ and $P_t\supseteq BJ^{-1}$.
Summing up, we get
$I_t=(I_vM)_t=(HJNP)_t=((HN)(JP))_t$, $(HN)_t\supseteq (AHH^{-1})_t=A_t$ and $(JP)_t\supseteq (BJJ^{-1})_t=B_t$.
\end{proof}

\begin{lemma}\label{332}  If D is an integrally closed domain and $I$ a nonzero ideal of $D[X]$ such that $I_v=D[X]$, then $I\cap D \neq 0$.
\end{lemma}
\begin{proof} Assume that  $I \cap D = 0$. By \cite[Theorem 2.1]{AKZ}, there exist $f\in D[X]-\{0\}$ and $a\in D-\{0\}$ such that $J:=(a/f)I \subseteq D[X]$ and $J\cap D \neq 0$.  We get  $(a/f)D[X]=(a/f)I_v \subseteq D[X]$, hence  $a/f \in D[X]$ and thus $a/f\in D$, because $a\in D-\{0\}$. We get $J\subseteq I$  which is a contradiction because $J\cap D  \neq 0$   and   $I \cap D = 0$.
\end{proof}

\begin{proposition} \label{331}
A domain $D$ is $t$-sharp if and only if $D[X]$ $t$-sharp.
\end{proposition}
\begin{proof} $(\Rightarrow).$ Set $\Omega=D[X]$. Let $I$,$A$,$B$ be nonzero ideals of $\Omega$ such that  $I\supseteq AB$.
By Proposition \ref{3x}, $D$ is pre-Krull, so $D[X]$ is pre-Krull, cf. \cite[Theorem 3.3]{L}. Applying Proposition \ref{3x} for $\Omega$, we can assume that $I_v=\Omega$. Changing $A$ by $I+A$ and $B$ by $I+B$, we may assume that $A,B \supseteq I$. By Lemma \ref{332} we get $I\cap D\neq 0$, so $A\cap D\neq 0$ and $B\cap D\neq 0$. By \cite[Theorem 3.2]{AKZ}, we have $I_t=I'_t\Omega$, $A_t=A'_t\Omega$ and $B_t=B'_t\Omega$ for some nonzero ideals $I',A'$ and $B'$ of D. From $I\supseteq AB$, we get $I'_t\Omega=I_t\supseteq A_tB_t=(A'_tB'_t)\Omega$, hence $I'_t\supseteq A'B'$.
As $D$ is $t$-sharp, there exist nonzero ideals $H$ and $J$ of $D$ such that $I'_t=(HJ)_t$, $H_t\supseteq  A'$ and $J_t\supseteq  B'$.
Hence $I_t=(HJ\Omega)_t$, $(H\Omega)_t\supseteq  A$ and $(J\Omega)_t\supseteq  B$. $(\Leftarrow).$ Let $I$,$A$,$B$ be nonzero ideals of $D$ such that  $I\supseteq AB$. As $D[X]$ $t$-sharp, there exist nonzero ideals $H$ and $J$ of $\Omega$ such that $(I\Omega)_t=(HJ)_t$, $H_t\supseteq  A$ and $J_t\supseteq  B$.
Since  $H_t\supseteq A$ and $A\neq 0$, we derive that $H_t\cap D\neq 0$,  hence $H_t=(M\Omega)_t$ for some nonzero ideal $M$ of $D$, cf. \cite[Theorem 3.2]{AKZ}. Similarly, $J_t=(N\Omega)_t$ for some nonzero ideal $N$ of $D$. Combining the relations above, we get $I_t=(MN)_t$, $M_t\supseteq  A$ and $N_t\supseteq  B$.
\end{proof}

\begin{remark}
Notice that we do not have a ``$d$-analogue'' of  Proposition \ref{331} because  a sharp domain  has dimension $\leq 1$ (see  \cite[Theorem 11]{ADE}). But remark that we do  have a ``$v$-analogue'' of  Proposition \ref{331}. Indeed, a domain $D$ is $v$-sharp  if and only if $D$ is completely integrally closed
(cf. Corollary \ref{100}) and  $D$ is completely integrally closed if and only if so is $D[X]$. Similarly, $D$ is $v$-sharp  if and only if the power series ring $D[[X]]$ is $v$-sharp.
\end{remark}

Denote by $N_v$ the multiplicative set of $D[X]$ consisting of all nonzero polynomials $a_0+a_1X+\cdots +a_nX^n$ such that $(a_0,a_1,...,a_n)_v=D$. The ring $D[X]_{N_v}$ was studied in \cite{Kg}.

\begin{proposition}\label{133}
A domain $D$ is $t$-sharp if and only if $D[X]_{N_v}$ is sharp.
\end{proposition}
\begin{proof}
If $D$ is $t$-sharp, then $D$ is a PVMD, cf. Proposition  \ref{11}.
If $D[X]_{N_v}$ is sharp, then $D[X]_{N_v}$ is a Prufer domain (cf. \cite[Theorem 11]{ADE}) hence $D$ is a PVMD, cf. \cite[Theorem 3.7]{Kg}. So we may assume from the beginning that $D$ is a PVMD. Note that the $t$-sharp  property of $D$ is in fact a property of the ordered monoid of all integral $t$-ideals of $D$ under the $t$-multiplication. Similarly, the sharp  property of $D[X]_{N_v}$ is  a property of the ordered monoid of all integral ideals of $D$ under the usual multiplication. Since $D$ is a PVMD, these two monoids are isomorphic (cf. \cite[Theorem 3.14]{Kg}), so the proof is complete.
\end{proof}

We end our paper with a (partial) power series analogue of Proposition \ref{133}. A lemma is in order.

\begin{lemma}\label{101}
Let $D\subseteq E$ be a domain extension and every ideal of $E$ is extended from $D$. If $D$ is sharp then $E$ is also sharp.
\end{lemma}
\begin{proof}
Let $I,A,B$ be nonzero ideals  of $D$ such that $IE\supseteq ABE$. Then $C=IE \cap D \supseteq  AB$. As $D$ is sharp, we have $C=HJ$ with $H,J$ ideals of $D$ such that  $H\supseteq A$ and $J\supseteq B$.  We get $IE=CE=HJE$,  $HE\supseteq AE$ and  $J\supseteq BE$.
\end{proof}

Let $D$ be a $t$-sharp domain which is not a field. 
By Proposition \ref{11}, $D$ is PVMD with $t$-dimension one. Hence \cite[Proposition 3.3]{AK2} shows that $c(fg)_t=(c(f)c(g))_t$ (thus $c(fg)_v=(c(f)c(g))_v$)
for every $f,g\in D[[X]]-\{0\}$,  where $c(f)$ is the ideal generated by the coefficients of $f$.
Then $N'_v=\{ f\in D[[X]]-\{0\}\mid c(f)_v=D\}$
is a multiplicative subset of the power series ring $D[[X]]$.
 The fraction ring $D[[X]]_{N'_v}$ was studied in \cite{AK2} and \cite{L}. Note that $D\subseteq D[X]_{N_v}\subseteq D[[X]]_{N'_v}$, where $N_v=\{ f\in D[X]-\{0\}\mid c(f)_v=D\}$.

\begin{proposition}\label{1024}
If  $D$ is a $t$-sharp domain, then $D[[X]]_{N'_v}$ is sharp and every ideal of $D[[X]]_{N'_v}$ is extended from $D$.
\end{proposition}
\begin{proof} 
We may assume that $D$ is not a field. 
By Proposition \ref{3x}, $D$ is a pre-Krull domain (alias $(t,v)$-Dedekind domain). As seen in the paragraph preceding this proposition, $c(fg)_v=(c(f)c(g))_v$
for every $f,g\in D[[X]]-\{0\}$. By  \cite[Theorem 4.3]{L}
 it follows that every ideal of $D[[X]]_{N'_v}$ is extended from $D$, then, a fortiori, from $D[X]_{N_v}$.
 By Proposition \ref{133} it follows that $D[X]_{N_v}$ is a sharp domain, hence so is $D[[X]]_{N'_v}$, cf. Lemma \ref{101}. 
\end{proof}

\begin{corollary}\label{307}
Let  $D$   be a  $t$-sharp domain. Then  $D[[X]]_{N'_v}$ is a faithfully flat $D[X]_{N_v}$-module and the extension map 
$I\mapsto ID[[X]]_{N'_v}$ 
 is a bijection from the set of  ideals of $D[X]_{N_v}$ to the set of  ideals of $D[[X]]_{N'_v}$.
\end{corollary}
\begin{proof}
Set $E=D[X]_{N_v}$ and $F=D[[X]]_{N'_v}$.
By the proof of Proposition \ref{133} it follows that $E$ is a Prufer domain. Hence 
$F$ is a flat $E$-module because over a Prufer domain every torsion-free module is flat. 
We show that every proper ideal of $E$ extends to a proper ideal of $F$. 
Let $J$ be a proper nonzero ideal of $E$. By \cite[Theorem 3.14]{Kg}, $J=IE$ for some ideal $I$ of $D$ such that $I_t\neq D$. Assume that $JF=F$. Then $IF=JF=F$, so $ID[[X]]$ contains some power series $f$ with $c(f)_v=D$. Write $f=a_1f_1+...+a_nf_n$ with $a_i \in I$ and $f_i \in D[[X]]$. Then $D=c(f)_v \subseteq  (a_1,...,a_n)_v\subseteq I_t$, so $I_t=D$, a contradiction. As $F$ is a flat $E$-module and every proper ideal of $E$ extends to a proper ideal of $F$, it follows that $F$ is a faithfully flat $E$-module, cf. \cite[Exercise 16, page 45]{AM}. In particular, $HF\cap E=H$ for each ideal $H$ of $E$. 
By Proposition \ref{1024},  every ideal of $F$ is extended from $D$ and hence from $E$ because $D\subseteq  E\subseteq  F$.
Combining these two facts, it follows that the extension map  $I\mapsto IF$ is a bijection from the set of  ideals of $E$ to the set of  ideals of $F$.
\end{proof}

{\bf Acknowledgements.} The first author was partially supported by an HEC (Higher Education Commission, Pakistan)  grant. The second author gratefully acknowledges the warm
hospitality of the Abdus Salam School of Mathematical Sciences GCU Lahore during his visits in 2006-2011. The third author was supported by UEFISCDI, project number 83/2010, PNII-RU code TE\_46/2010, program Human Resources.
\\[7mm]


\begin{thebibliography}{9999}

\bibitem{AD} Z. Ahmad and T. Dumitrescu,
Almost quasi-Schreier domains,
to appear in Comm. Algebra.

\bibitem{ADE} Z. Ahmad, T. Dumitrescu and M. Epure,  A Schreier domain type condition, (submitted), arXiv:1112.0132v1.

\bibitem{A} { D.D. Anderson,}
{Star-operations induced by overrings,}
Comm. Algebra {\bf 16} (1988), 2535-2553.

\bibitem{AA} { D.D. Anderson and D.F. Anderson,}
Generalized GCD domains,
Comment. Math. Univ. St. Pauli {\bf 28} (1979), 215-221.


\bibitem{ACl} D.D. Anderson and S. Clarke,
Star-operations that distribute over finite intersections,
Comm. Algebra {\bf 32} (2005), 2263-2274.

\bibitem{AC} { D.D. Anderson and S.J. Cook,}{Two star operations and their indused latices,} Comm. Algebra {\bf 28} (2000), 2461-2475.

\bibitem{ADZ} D.D. Anderson, T. Dumitrescu and M. Zafrullah,
Quasi-Schreier domains II, Comm. Algebra  {\bf 35} (2007),  2096-2104.

\bibitem{AK} { D.D. Anderson and B.G. Kang,}
Pseudo-Dedekind  domains and divisorial ideals in $R[X]_T$,
J. Algebra {\bf 122} (1989), 323-336.

\bibitem{AK2} { D.D. Anderson and B.G. Kang,}
Formally integrally closed domains and the rings $R((X))$ and $R\{\{X\}\}$,
J. Algebra {\bf 200} (1998), 347-362.


\bibitem{AKZ} {D.D. Anderson, D.J. Kwak and M. Zafrullah,}
{Agreeable domains,} Comm.  Algebra { 23} (1995), 4861-4883.


\bibitem{AZ} { D.D. Anderson and M. Zafrullah,}
{Independent locally-finite intersections of localizations,} Houston J. Math. {\bf 25} (1999), 433-452.

\bibitem{AM}  {M.F. Atiyah and I.G. Macdonald,} 
{\em Introduction to Commutative Algebra}, Addison-Wesley, Reading, 1969.

\bibitem{B} V. Barucci,
Mori domains,
in {\em Non-Noetherian commutative ring theory},  57-73, Math. Appl., {\bf 520}, Kluwer Acad. Publ., Dordrecht,  2000.


\bibitem{C1} {G.W. Chang,} {Strong Mori domains and the ring $D[x]_{N_v}$,}
J. Pure Appl. Algebra {\bf 197} (2005), 3669-3686.

\bibitem{C} P.M. Cohn,  Bezout rings and their subrings,
{Proc. Cambridge Philos. Soc.} {\bf 64} (1968), 251-264.

\bibitem{D} J. Dieudonne, Sur la theorie de la divisibilite, Bull. Soc. Math. France {\bf 69} (1941), 133-144.

\bibitem{DK} T. Dumitrescu and W. Khalid,
Almost-Schreier domains,
{Comm. Algebra.} {\bf 38} (2010), 2981-2991.

\bibitem{DM} T. Dumitrescu and R. Moldovan,
Quasi-Schreier
domains,
{Math. Reports} {\bf 5} (2003), 121-126.


\bibitem{DZ} T. Dumitrescu and M. Zafrullah,
Characterizing domains of finite $*$-character,
J. Pure Appl. Algebra {\bf 214} (2010), 2087-2091.

\bibitem{DZ1} T. Dumitrescu and M. Zafrullah,
t-Schreier domains,
Comm.  Algebra,  {\bf 39} (2011), 808-818.


\bibitem{EFP} S. El Baghdadi, M. Fontana and G. Picozza, Semistar Dedekind domains,
J. Pure Appl. Algebra {\bf 193} (2004), 27-60.

\bibitem{FJS} M. Fontana, P. Jara and E. Santos,
Prufer $*$-multiplication domains and semistar operations,
J. Algebra Appl. {\bf 2} (2003), 21-50.


\bibitem{F}  {R. Fossum,}
{\em The divisor class group of a Krull domain},
Springer, New York, 1973.

\bibitem{G} R. Gilmer, {\em Multiplicative Ideal Theory}, Marcel Dekker, New York, 1972.


\bibitem{Gr} { M. Griffin,}
{Some results on v-multiplication rings,}
Canad. J. Math. {\bf 19} (1967), 710-722.



\bibitem{H} F. Halter-Koch, {\em Ideal Systems: an Introduction to Multiplicative Ideal Theory}, Marcel Dekker, New York 1998.

\bibitem{HZ} E. Houston and M. Zafrullah, Integral domains in
which each t-ideal is divisorial, Mich. Math. J.
{\bf 35} (1988), 291-300.

\bibitem{Kg} {B.G. Kang,}
{Pr\"ufer v-multiplication domains and the ring $R[X]_{N_v}$,}
J. Algebra
{\bf 123} (1989), 151-170.

\bibitem{Kg1} {B.G. Kang,}
On the converse of a well-known fact about Krull domains,
J. Algebra {\bf 124} (1989), 284-299.


\bibitem{L} {Q. Li,}
$(t,v)$-Dedekind domains and the ring $R[X]_{N_v}$,
Results. Math.  {\bf 59} (2011), 91-106.



\bibitem{MR} S. McAdam and D.E. Rush,
Schreier Rings,
Bull. London Math. Soc. {\bf 10} (1978), 77-80.


\bibitem{Z} { M. Zafrullah,}
On generalized Dedekind  domains,
Mathematika  {\bf 33} (1986), 285-295.



\bibitem{Zps} M. Zafrullah,
On a property of pre-Schreier domains,
Comm. Algebra {\bf 15} (1987), 1895-1920.

\bibitem{Zac} M. Zafrullah,
Ascending chain conditions and star operations,
Comm. Algebra {\bf 17} (1989), 1523-1533.


\end{thebibliography}
\end{document}